\documentclass[11pt]{amsart}

\usepackage{verbatim,esint}
\usepackage{geometry}
\usepackage[colorlinks=true,urlcolor=blue,
citecolor=red,linkcolor=blue,linktocpage,pdfpagelabels,
bookmarksnumbered,bookmarksopen]{hyperref}

\usepackage[hyperpageref]{backref}

\usepackage{lmodern}

\newtheorem{theorem}{Theorem}[section]
\newtheorem{lemma}[theorem]{Lemma}
\newtheorem{proposition}[theorem]{Proposition}

\newtheorem{corollary}[theorem]{Corollary}
\theoremstyle{definition}
\newtheorem{remark}[theorem]{Remark}
\newtheorem{definition}[theorem]{Definition}

\newtheorem*{ack}{Acknowledgements}

\numberwithin{equation}{section}

\title[Poincar\'e inequalities via Optimal Transport]{A note on some\\ Poincar\'e inequalities on convex sets\\ by Optimal Transport methods}
\author[Brasco]{Lorenzo Brasco}
\author[Santambrogio]{Filippo Santambrogio}

\address[L.\ Brasco]{Dipartimento di Matematica e Informatica
\newline\indent
Universit\`a degli Studi di Ferrara,
44121 Ferrara, Italy}
\address{{\it and } Institut de Math\'ematiques de Marseille
\newline\indent
Aix-Marseille Universit\'e,
Marseille, France}
\email{lorenzo.brasco@unife.it}

\address[F.\ Santambrogio]{Laboratoire de Math\'ematiques d'Orsay 
\newline\indent 
Univ. Paris-Sud, CNRS, Universit\'e Paris-Saclay, 91405 Orsay Cedex, France}
\email{filippo.santambrogio@math.u-psud.fr}


\begin{document}

\begin{abstract}
We show that a class of Poincar\'e-Wirtinger inequalities on bounded convex sets can be obtained by means of the dynamical formulation of Optimal Transport. This is a consequence of a more general result valid for convex sets, possibly unbounded.
\end{abstract}

\maketitle

\begin{center}
\begin{minipage}{11cm}
\small
\tableofcontents
\end{minipage}
\end{center}

\section{Introduction}

\subsection{Overview}
Let $1<p<\infty$ and $0<r<\infty$. For an open set $\Omega\subset\mathbb{R}^N$, we introduce the Sobolev spaces
\[
\mathrm{W}^{1,p}_{r}(\Omega):=\left\{\phi\in L^{r}(\Omega)\,:\, \nabla \phi\in L^p(\Omega;\mathbb{R}^N)\right\},
\]
and
\[
\ddot{\mathrm{W}}^{1,p}_{r}(\Omega):=\left\{\phi\in \mathrm{W}^{1,p}_{r}(\Omega)\,:\, \int_\Omega |\phi|^{r-1}\,\phi\,dx=0\right\}.
\]
In the particular case $r=p$, we will omit to indicate it and simply write $\mathrm{W}^{1,p}(\Omega)$ and $\ddot{\mathrm{W}}^{1,p}(\Omega)$.
\par
The aim of this note is to prove some functional inequalities for the space $\ddot{\mathrm{W}}^{1,p}_{r}(\Omega)$, by means of Optimal Transport techniques. The use of Optimal Transport to prove functional and geometric inequalities is nowadays classical. We are not concerned here with geometric inequalities, thus we only refer to Sections 2.5.3 and 7.4.2 of \cite{OTAM} for a brief discussion on the subject (in particular on the isoperimetric and the Brunn-Minkowski inequalities). As for functional inequalities obtained via Optimal Transport techniques, which is the main concern of this paper, after the fundamental paper \cite{CENV} the literature on the subject is now quite rich. In addition to \cite{CENV}, we encourage the reader to look in details into the papers \cite{AGK,Cor,MV1,MV2} and \cite{Na}, for example.
\par
However, it is useful to observe that most of these papers use the geometric properties of the optimal transport map as a tool to obtain a clever change-of-variable. This is indeed the case for the transport-based proof of the isoperimetric, Sobolev and Gagliardo-Nirenberg inequalities. We could say that they are based on the ``statical'' version of Optimal Transport problems. 
\par
On the contrary, the proof that we propose here is based on the ``dynamical'' counterpart of Optimal Transport (the so-called {\it Benamou-Brenier formula}, see \cite{BenBre}) and on {\it displacement convexity} considerations, see \cite{McC}. In this respect, it can be more suitably compared to the transport-based proof of the Brunn-Minkowski inequality. 
\par
It is also useful to remark that while the above cited papers deal with functional inequalities which are invariant for the trasformation $\phi\mapsto |\phi|$, such as Sobolev and Gagliardo-Nirenberg ones, this is not the case here. Indeed, if a function $\phi$ belongs to our space $\ddot{\mathrm{W}}^{1,p}_{r}(\Omega)$, then $|\phi|\not\in \ddot{\mathrm{W}}^{1,p}_{r}(\Omega)$. Thus, in order to prove our main result (see Theorem \ref{thm:mother} below), we can not reduce to the case of positive functions and then use an optimal transport to transform any positive function $\phi$ into an extremal of the relevant functional inequality, as in \cite{CENV}. Roughly speaking, what we do is to perform an optimal transport between the positive and negative parts $\phi_+$ and $\phi_-$ (suitably renormalized). %In doing this, we will use transports which are optimal with respect to the cost
\par
Our proof has some points in common
with the one presented by Rajala in \cite{Ra}, which is valid in general metric measure spaces under {\it Ricci curvature conditions}. Indeed, it is well-known that Ricci curvature conditions are linked to the displacement convexity of suitable functionals (see for instance the work \cite{LV} by Lott and Villani, to which \cite{Ra} is inspired). However, even if the result of \cite[Theorem 1.1]{Ra} holds in a much more general setting, we stress that the tools used in \cite{Ra} are not the same as ours. Moreover, the result of \cite{Ra} only concerns with Poincar\'e inequalities on balls in the case $q=1$ (with our notation below).

\subsection{Main result}
In order to neatly present the main result, we first need to recall some basic definitions and notations. 
\par
We indicate by $\mathcal{P}(\Omega)$ the set of all Borel probability measures over $\Omega$. Then for $1<m<\infty$, we define
\begin{equation}
\label{momento!}
\mathcal{P}_m(\Omega)=\left\{\mu\in \mathcal{P}(\Omega)\, :\, \int_\Omega |x|^m\,d\mu<\infty \right\},
\end{equation}
i.e. the set of probability measure over $\Omega$ with finite moment of order $m$.
For every $\mu,\nu\in \mathcal{P}_m(\Omega)$ their {\it $m-$Wasserstein distance} is defined through the optimal transport problem
\[
W_m(\mu,\nu)=\left(\min_{\gamma\in\Pi(\mu,\nu)}\int_{\Omega\times\Omega} |x-y|^m\,d\gamma\right)^\frac{1}{m}.
\]
Here $\Pi(\mu,\nu)\subset \mathcal{P}(\Omega\times\Omega)$ is the set of {\it transport plans}, i.e. the probability measures on the product space $\Omega\times\Omega$ such that
\[
\gamma(A\times\Omega)=\mu(A)\qquad \gamma(\Omega\times B)=\nu(B),\qquad \mbox{ for every } A,B\subset\Omega \mbox{ Borel sets}.
\]
In what follows, we will note by $\mathcal{L}^N$ the $N-$dimensional Lebesgue measaure. For a function $f\in L^1$, the writing
\[
\mu=f\cdot\mathcal{L}^N,
\] 
will indicate the Radon measure which is absolutely continuous with respect to $\mathcal{L}^N$ and whose Radon-Nikodym derivative is given by $f$.
\par
In this note we prove the following scaling invariant inequality, which is valid for general convex sets.
\begin{theorem}
\label{thm:mother}
Let $1<p<\infty$ and $1<q< p$. Let $\Omega\subset\mathbb{R}^N$ be an open convex set. For every $\phi\in \ddot{\mathrm{W}}^{1,p}_{q-1}(\Omega)$ such that
\[
\int_\Omega |x|^\frac{p}{p-q}\,|\phi|^{q-1}\,dx<\infty,
\]
we define the two probability measures $\rho_0,\rho_1\in\mathcal{P}_{p/(p-q)}(\Omega)$
\[
\rho_0=\frac{|\phi|^{q-2}\,\phi_+}{\displaystyle\int_\Omega |\phi|^{q-2}\,\phi_+\,dx}\cdot \mathcal{L}^N\qquad \mbox{ and }\qquad \rho_1=\frac{|\phi|^{q-2}\,\phi_-}{\displaystyle\int_\Omega |\phi|^{q-2}\,\phi_-\,dx}\cdot \mathcal{L}^N.
\] 
Then there holds
\begin{equation}
\label{filippica}
\left(\int_\Omega |\phi|^q\,dx\right)^{p-q+1}\le \frac{\left(W_\frac{p}{p-q}(\rho_0,\rho_1)\right)^p}{2^{p-1}}\,\int_\Omega |\nabla \phi|^p\,dx\,\left(\displaystyle\int_\Omega |\phi|^{q-1}\,dx\right)^{p-q}.
\end{equation}
\end{theorem}
The proof of this result is postponed to Section \ref{sec:3}. We point out that inequality \eqref{filippica} in turn implies a handful of Poincar\'e-type inequalities with explicit constants. The reader is invited to jump directly to Section \ref{sec:4} in order to discover them. In particular, as a corollary we can obtain a lower bound for the first non-trivial Neumann eigenvalue of the $p-$Laplacian, see Corollary \ref{coro:PW}. This can be seen as a weak version of the {\it Payne-Weinberger inequality} (see \cite{Be,FNT,PW}): though the explicit constant we get is not optimal, we believe the method of proof to be of independent interest.
\begin{remark}
We point out that the hypothesis $\phi\in L^q(\Omega)$ {\it is not needed} in Theorem \ref{thm:mother}. Rather, inequality \eqref{filippica} permits to show that on a convex set, functions in $\ddot{\mathrm{W}}^{1,p}_{q-1}(\Omega)$ and with finite moment of order $p/(p-q)$ are automatically in $L^q(\Omega)$.
\end{remark}

\begin{ack}
The authors wish to thank Cristina Trombetti for pointing out the references \cite{ENT} and \cite{FNT}.
This work has been partially supported by the Gaspard Monge Program for Optimization (PGMO), created by EDF and the Jacques Hadamard Mathematical Foundation, through the research contract MACRO, and by the ANR through the contract ANR- 12-BS01-0014-01 GEOMETRYA.
\end{ack}

\section{Preliminaries}

\subsection{An embedding result} We will need a couple of basic inequality for Sobolev spaces in bounded sets. The proofs are standard, but we give it for the reader's convenience. The values of the constants appearing in the inequalities below will have no bearing in what follows.
\begin{lemma}
Let $1<p<\infty$ and let $\Omega\subset\mathbb{R}^N$ be an open connected and bounded set, with Lipschitz boundary. Then for every $\phi\in \mathrm{W}^{1,p}(\Omega)$ we have
\begin{equation}
\label{giusti}
\int_\Omega |\phi|^p\,dx\le C\,\frac{|\Omega|}{|A_\phi|}\,\int_\Omega |\nabla \phi|^p\,dx,\qquad A_\phi:=\{x\in\Omega\, :\, |\phi(x)|=0\},
\end{equation}
for some $C=C(N,p,\Omega)>0$.
\end{lemma}
\begin{proof}
The proof is an adaptation of that of \cite[Theorem 3.16]{Gi}. We first observe that if we indicate by $\overline \phi_\Omega$ the mean of $\phi$ over $\Omega$, then
\[
|A_\phi|\,|\overline\phi_\Omega|^p=\int_{A_\phi} |\overline\phi_\Omega|^p\,dx=\int_{A_\phi} |\phi-\overline\phi_\Omega|^p\,dx\le \int_{\Omega} |\phi-\overline\phi_\Omega|^p\,dx.
\] 
By using this information, with elementary manipulations we then get
\[
\int_\Omega |\phi|^p\,dx\le 2^{p-1}\, \int_\Omega |\phi-\overline\phi_\Omega|^p\,dx+2^{p-1}\, \frac{|\Omega|}{|A_\phi|}\,\int_\Omega |\phi-\overline \phi_\Omega|^p\,dx.
\]
We can conclude by applying Poincar\'e inequality for functions with vanishing mean, see for example \cite[Theorem 3.14]{Gi}.
\end{proof}
The next interpolation inequality for the Sobolev space $\mathrm{W}^{1,p}_r(\Omega)$ will be useful.  
\begin{lemma}
\label{lm:dentro}
Let $1<p<\infty$ and $0<r< p$. Let $\Omega\subset\mathbb{R}^N$ be a open connected and bounded set with Lipschitz boundary. Then $\mathrm{W}^{1,p}_{r}(\Omega)\subset L^p(\Omega)$. More precisely, for every $\phi\in \mathrm{W}^{1,p}_r(\Omega)$ we have
\[
\int_\Omega |\phi|^p\,dx\le C\, \int_\Omega |\nabla \phi|^p\,dx+C\left(\int_\Omega |\phi|^r\,dx\right)^\frac{p}{r},
\]
for some $C=C(N,p,\Omega)>0$.
\end{lemma}
\begin{proof}
Given $\phi\in \mathrm{W}^{1,p}_{r}(\Omega)$, for every $t>0$ and $M>0$ we define
\[
 \phi_{t}(x)=(|\phi(x)|-t)_+\qquad\mbox{ and }\qquad \phi_{t,M}(x)=\min\{\phi_t(x),\, M\}.
\]
The function $\phi_{t,M}$ belongs to $\mathrm{W}^{1,p}(\Omega)$ and by Chebyshev's inequality
\begin{equation}
\label{zeri}
|A_{t,M}|:=\big|\{x\in \Omega\, :\,\phi_{t,M}(x)\not=0\}\big|\le \frac{1}{t^{r}}\,\int_\Omega |\phi|^{r}\,dx.
\end{equation}
From \eqref{giusti} we get
\[
\int_\Omega |\phi_{t,M}|^p\,dx\le C\, \left(\frac{|\Omega|}{|\Omega\setminus A_{t,M}|}\right)\,\int_\Omega |\nabla \phi_{t,M}|^p\,dx,
\]
and observe that from \eqref{zeri}
\[
\frac{|\Omega|}{|\Omega\setminus A_{t,M}|}=\frac{|\Omega|}{|\Omega|-|A_{t,M}|}\le \frac{1}{2},\qquad \mbox{ if we choose } t=\left(\frac{2}{|\Omega|}\right)^{1/r}\,\|\phi\|_{L^r(\Omega)}.
\]
We thus obtain
\[
\int_\Omega |\phi_{t,M}|^p\,dx\le \frac{C}{2}\, \int_\Omega |\nabla \phi|^p\,dx.
\]
It is now possible to take the limit as $M$ goes to $\infty$, thus getting by Fatou's Lemma 
$$
\int_\Omega |\phi_{t}|^p\,dx\le \frac{C}{2}\, \int_\Omega |\nabla \phi|^p\,dx.
$$
By recalling the choice of $t$ and observing that $|\phi|\leq t+\phi_{t}$, we get the desired conclusion.
\end{proof}
\subsection{Some tools from Optimal Transport}
We recall a couple of standard result in Optimal Transport, that will be needed for the proof of the main result. For more details, the reader is invited to refer to classical monographs such as \cite{AGS} or \cite{Vi}, or to the more recent one \cite{OTAM}.
\begin{definition}
The {\it $m-$Wasserstein space over $\Omega$} is the set $\mathcal{P}_m(\Omega)$ defined in \eqref{momento!}, equipped with Wasserstein the distance $W_m$. This metric space will be denoted by $\mathbb{W}_m(\Omega)$.
\end{definition}
The first important tool we need is a characterization of geodesics in the Wasserstein space. This is essentially a refined version of the celebrated Benamou-Brenier formula, firstly introduced in \cite{BenBre}.
The proof can be found in \cite[Theorem 5.14 \& Proposition 5.30]{OTAM}.
\begin{proposition}[Wasserstein geodesics]
\label{thm:geodesics}
Let $1<m<\infty$ and let $\Omega\subset\mathbb{R}^N$ be an open bounded convex set. Let $\rho_0,\rho_1\in \mathbb{W}_m(\Omega)$, then there exists an absolutely continuous curve $(\mu_t)_{t\in[0,1]}$ in the Wasserstein space $\mathbb{W}_m(\Omega)$ and a vector field $\mathbf{v}_t\in L^m(\Omega;\mu_t)$ such that
\begin{itemize}
\item $\mu_0=\rho_0$ and $\mu_1=\rho_1$;
\vskip.2cm
\item the continuity equation 
\[
\left\{\begin{array}{rclc}
\partial_t \mu_t+\mathrm{div}(\mathbf{v}_t\,\mu_t)&=&0,& \mbox{ in }\Omega,\\
\langle \mathbf{v}_t,\nu_\Omega\rangle& = &0,& \mbox{ on }\partial\Omega
\end{array}
\right.
\]
holds in distributional sense, i.e. for every $\phi\in C^1([0,1]\times\overline\Omega)$ there holds
\[
\int_0^1\int_\Omega \partial_t \phi\,d\mu_t\,dt+\int_0^1\int_\Omega \langle \nabla \phi,\mathbf{v}_t\rangle\,d\mu_t\,dt=\int_\Omega \phi(1,\cdot)\,d\rho_1-\int_\Omega \phi(0,\cdot)\,d\rho_0;
\]
\item we have
\[
\int_0^1\|\mathbf{v}_t\|_{L^m(\Omega; \mu_t)}\,dt= W_m(\rho_0,\rho_1).
\]
\end{itemize}
\end{proposition}
The other expedient result from Optimal Transport we need is the following convexity property of $L^q$ norms. For $m=2$, the following one is a particular case of a result by McCann, see \cite{McC}. The proof can be found, for instance, in \cite[Theorem 7.28]{OTAM}.
\begin{proposition}[Geodesic convexity of $L^p$ norms]
\label{convexity}
Let $1<m<\infty$ and let $\Omega\subset\mathbb{R}^N$ be an open bounded convex set. Let $\rho_0=f_0\cdot\mathcal{L}^N$ and $\rho_1=f_1\cdot\mathcal{L}^N$ be two probability measures on $\Omega$, such that $f_0,f_1\in L^q(\Omega)$ for some $1\le q\le \infty$. If $(\mu_t)_{t\in[0,1]}\subset\mathbb{W}_m(\Omega)$ is the curve of Theorem \ref{thm:geodesics}, then we have
\[
\mu_t=f_t\cdot\mathcal{L}^N\qquad \mbox{ and }\qquad \|f_t\|_{L^q(\Omega)}\le \left((1-t)\,\|f_0\|^q_{L^q(\Omega)}+t\,\|f_1\|^q_{L^q(\Omega)}\right)^\frac{1}{q},\quad t\in[0,1].
\]
\end{proposition}

\section{Proof of the main result}
\label{sec:3}

\subsection{An expedient estimate}
We first need the following preliminary result. The idea of the proof is similar to that of \cite[Proposition 2.6]{Lo} and \cite[Lemma 3.5]{MRCS}, though the final outcome is different. We also cite the short unpublished note \cite{Pey} containing interesting uniform estimates on these topics.
\begin{lemma}
Let $1<q<p<\infty$ and let $\Omega\subset\mathbb{R}^N$ be an open and bounded convex set. For every $\phi\in \mathrm{W}^{1,p}(\Omega)$ and every $f_0,f_1\in L^{q'}(\Omega)$ such that 
\[
\int_\Omega f_0\,dx=\int_\Omega f_1\,dx=1,\qquad f_0,f_1\ge0,
\]
we have
\begin{equation}
\label{basic}
\begin{split}
\int_\Omega \phi\, (f_1-f_0)\, dx\le W_\frac{p}{p-q}(\rho_0,\rho_1)\,\|\nabla \phi\|_{L^{p}(\Omega)}\,\left(\frac{\|f_0\|^{q'}_{L^{q'}(\Omega)}+\|f_1\|^{q'}_{L^{q'}(\Omega)} }{2}\right)^\frac{q-1}{p},
\end{split}
\end{equation}
where
\[
\rho_i=f_i\cdot\mathcal{L}^N,\qquad i=0,1,
\]
\end{lemma}
\begin{proof}
Let us first suppose that $\phi\in C^1(\overline\Omega)$. In this case we clearly have $C^1(\overline\Omega)\subset \mathrm{W}^{1,p}(\Omega)$.
\par
Then, by using Propositions \ref{thm:geodesics} and \ref{convexity} with $\rho_0=f_0\cdot\mathcal{L}^N$ and $\rho_1=f_1\cdot\mathcal{L}^N$ and observing that $\phi$ does not depend on $t$, with the previous notation we can infer
\[
\begin{split}
\int_\Omega \phi\, (f_1-f_0)\, dx&=\int_0^1 \int_\Omega \langle \nabla\phi,\mathbf{v}_t\rangle\, f_t\,dx\, dt\\
&\le \left(\int_0^1\int_\Omega |\nabla \phi|^\frac{p}{q}\,f_t\,dx\, dt\right)^\frac{q}{p}\,\left(\int_0^1\int_\Omega |\mathbf{v}_t|^r\, f_t\,dx\,dt\right)^\frac{1}{r}\\
&\le\left(\int_0^1\int_\Omega |\nabla \phi|^{p}\,dx\, dt\right)^\frac{1}{p}\,\left(\int_0^1 \|f_t\|^{q'}_{L^{q'}(\Omega)}\,dt\right)^\frac{q-1}{p}\,W_r(\rho_0,\rho_1),
\end{split}
\]
where for notational simplicity we set $r:=p/(p-q)$.
Observe that the last term is finite, since $f_t\in L^{q'}(\Omega)$ and its $L^{q'}$ norm is integrable in time, thanks to Proposition \ref{convexity}.
\par
Since $\phi$ does not depend on $t$, 
from the previous estimate we get in particular
\[
\int_\Omega \phi\, (f_1-f_0)\, dx\le W_r(\rho_0,\rho_1)\,\|\nabla \phi\|_{L^{p}(\Omega)}\,\left(\int_0^1 \|f_t\|^{q'}_{L^{q'}(\Omega)}\,dt\right)^\frac{q-1}{p}.
\]
We now observe that by Proposition \ref{convexity}
\[
\begin{split}
\int_0^1 \|f_t\|^{q'}_{L^{q'}(\Omega)}\,dt&\le \int_0^1 \left[\|f_0\|^{q'}_{L^{q'}(\Omega)} +t\,\left(\|f_1\|^{q'}_{L^{q'}(\Omega)}-\|f_0\|^{q'}_{L^{q'}(\Omega)} \right)\right]\,dt\\
&=\frac{\|f_0\|^{q'}_{L^{q'}(\Omega)}+\|f_1\|^{q'}_{L^{q'}(\Omega)} }{2}.
\end{split}
\]
thus we obtain the desired estimate \eqref{basic}, for $\Omega$ bounded and $\phi\in C^1(\overline\Omega)$.
\par
Finally, we get the general case by using the density of $C^1(\overline\Omega)$ in $\mathrm{W}^{1,p}(\Omega)$, see \cite[Theorem 1, Section 1.1.6]{Ma}.
\end{proof}

\subsection{Proof of Theorem \ref{thm:mother}}
We divide the proof in two steps: we first prove the inequality for bounded convex sets and then consider the general case. For the sake of simplicity, we set again $r:=p/(p-q)$.
\vskip.2cm\noindent
\underline{\it Bounded convex sets.} Let $\phi\in \ddot{\mathrm{W}}^{1,p}_{q-1}(\Omega)\setminus\{0\}$, the hypothesis $\int_\Omega |\phi|^{p-2}\,\phi=0$ implies
\begin{equation}
\label{smezzo}
\int_\Omega |\phi|^{q-1}\,dx=2\int_\Omega |\phi|^{q-2}\,\phi_+\,dx=2\,\int_\Omega |\phi|^{q-2}\,\phi_-\,dx.
\end{equation}
By Lemma \ref{lm:dentro}, we have $\phi\in \mathrm{W}^{1,p}(\Omega)$ as well, thus we can now apply \eqref{basic} with the choices
\[
\rho_1=f_1\cdot\mathcal{L}^N:=\frac{|\phi|^{q-2}\,\phi_+}{\displaystyle\int_\Omega |\phi|^{q-2}\,\phi_+\,dx}\cdot\mathcal{L}^N \quad \mbox{ and }\quad \rho_0=f_0\cdot\mathcal{L}^N=\frac{|\phi|^{q-2}\,\phi_-}{\displaystyle\int_\Omega|\phi|^{q-2}\,\phi_-\,dx}\cdot\mathcal{L}^N.
\]
For the left-hand side of \eqref{basic}, by using \eqref{smezzo} we get
\[
\int_\Omega \phi\, (f_1-f_0)\, dx=2\,\frac{\displaystyle\int_\Omega |\phi|^q\,dx}{\displaystyle\int_\Omega |\phi|^{q-1}\,dx}.
\]
For the right-hand side of \eqref{basic}, we observe that again by \eqref{smezzo} and using that 
\[
|\phi|^{q-2}\,\phi_+=\phi_+^{q-1},\qquad |\phi|^{q-2}\,\phi_-=\phi_-^{q-1},
\]
we get
\[
\begin{split}
\|f_0\|^{q'}_{L^{q'}(\Omega)}+\|f_1\|^{q'}_{L^{q'}(\Omega)}&=\frac{\displaystyle\int_\Omega \left(|\phi|^{q-2}\,\phi_-\right)^\frac{q}{q-1}\,dx}{\displaystyle\left(\int_\Omega |\phi|^{q-2}\, \phi_-\,dx\right)^\frac{q}{q-1}}+\frac{\displaystyle\int_\Omega \left(|\phi|^{q-2}\,\phi_+\right)^\frac{q}{q-1}\,dx}{\displaystyle\left(\int_\Omega |\phi|^{q-2}\, \phi_+\,dx\right)^\frac{q}{q-1}}\\
%&= 2^\frac{q}{q-1}\,\frac{\displaystyle\displaystyle\int_\Omega \left(|\phi|^{q-2}\,\phi_-\right)^\frac{q}{q-1}+\displaystyle\int_\Omega \left(|\phi|^{q-2}\,\phi_+\right)^\frac{q}{q-1}\,dx}{\left(\displaystyle\int_\Omega |\phi|^{q-1}\,dx\right)^\frac{q}{q-1}}\\
&=2^\frac{q}{q-1}\,\frac{\displaystyle\displaystyle\int_\Omega |\phi|^q\,dx}{\left(\displaystyle\int_\Omega |\phi|^{q-1}\,dx\right)^\frac{q}{q-1}}.
\end{split}
\]
Then from \eqref{basic} we finally obtain
\[
\frac{\displaystyle\int_\Omega |\phi|^q\,dx}{\displaystyle\int_\Omega |\phi|^{q-1}\,dx}\le \frac{W_r(\rho_0,\rho_1)}{2^\frac{p-1}{p}}\,\left(\int_\Omega |\nabla \phi|^p\,dx\right)^\frac{1}{p}\,\frac{\displaystyle\left(\int_{\Omega} |\phi|^q\,dx\right)^\frac{q-1}{p}}{\left(\displaystyle\int_\Omega |\phi|^{q-1}\,dx\right)^\frac{q}{p}}.
\]
After a simplification, this proves the desired inequality \eqref{filippica} when $\Omega$ is a bounded set.
\vskip.2cm\noindent
\underline{\it General convex sets}.
Let us now assume that $\Omega$ is a generic open convex set and $\phi\in \ddot{\mathrm{W}}^{1,p}_{q-1}(\Omega)$. We can suppose that the origin belongs to $\Omega$, then for $k\in\mathbb{N}\setminus\{0\}$ we define
\[
\Omega_k=\{x\in\Omega\, :\, |x|<k\}\qquad \mbox{ and }\qquad \delta_k=\left(\frac{\displaystyle\int_{\Omega_k} |\phi_+|^{q-1}\,dx}{\displaystyle\int_{\Omega_k} |\phi_-|^{q-1}\,dx}\right)^{1/(q-1)}.
\]
Note that, at least for $k$ large, $\delta_k$ is well-defined, since 
\[
\lim_{k\to\infty} \int_{\Omega_k} |\phi_-|^{q-1}\,dx=\int_{\Omega} |\phi_-|^{q-1}\,dx, 
\]
and the last quantity is strictly positive, unless $\phi=0$ (in this case there would be nothing to prove).
\par
The function $\phi_k=\phi_+-\delta_k\,\phi_-$ belongs to $\ddot{\mathrm{W}}^{1,p}_{q-1}(\Omega_k)$, by construction. Moreover, since $\phi\in \ddot{\mathrm{W}}^{1,p}_{q-1}(\Omega)$, we have
\begin{equation}
\label{deltak}
\lim_{k\to\infty} \delta_k=1.
\end{equation}
We also set
\[
\rho_{1,k}:=\frac{|\phi_k|^{q-2}\,(\phi_k)_+}{\displaystyle\int_{\Omega_k} |\phi_k|^{q-2}\,(\phi_k)_+\,dx}\cdot\mathcal{L}^N=\frac{|\phi|^{q-2}\,\phi_+}{\displaystyle\int_{\Omega_k} |\phi|^{q-2}\,\phi_+\,dx}\cdot\mathcal{L}^N
\]
and
\[
\rho_{0,k}:=\frac{|\phi_k|^{q-2}\,(\phi_k)_-}{\displaystyle\int_{\Omega_k}|\phi_k|^{q-2}\,(\phi_k)_-\,dx}\cdot\mathcal{L}^N=\frac{|\phi|^{q-2}\,\phi_-}{\displaystyle\int_{\Omega_k}|\phi|^{q-2}\,\phi_-\,dx}\cdot\mathcal{L}^N.
\]
Since $\Omega_k$ is bounded, from the previous step we obtain
\begin{equation}
\label{alfinito}
\left(\int_{\Omega_k} |\phi_k|^q\,dx\right)^{p-q+1}\le \frac{\left(W_r(\rho_{0,k},\rho_{1,k})\right)^p}{2^{p-1}}\,\int_{\Omega_k} |\nabla \phi_k|^p\,dx\,\left(\displaystyle\int_{\Omega_k} |\phi_k|^{q-1}\,dx\right)^{p-q}.
\end{equation}
We now observe that
\[
\lim_{k\to\infty} W_r(\rho_{0,k},\rho_{1,k})=W_r(\rho_0,\rho_1).
\]
Indeed, it is enough to remark that we have $\rho_{i,k}\to\rho_i$ in $\mathbb{W}_r(\Omega)$ for $i=0,1$. This follows from the fact that the convergence in $\mathbb W_r$ is equivalent to the weak convergence plus the convergence of the moments of order $r$ (see for instance \cite[Theorem 5.11]{OTAM}). Both conditions are easily seen to hold true here.
\par
Moreover, by construction we have
\[
|\phi_k|^{q-1}\cdot 1_{\Omega_k}\le \left(\max\{1,\delta_k\}\right)^{q-1}\,|\phi|^{q-1}\,\cdot 1_{\Omega},
\]
and
\[
|\nabla \phi_k|^p\cdot 1_{\Omega_k}\le  \left(\max\{1,\delta_k\}\right)^p\,|\nabla \phi|^p\,\cdot 1_\Omega.
\]
If we use \eqref{deltak}, we can pass to the limit as $k$ goes to $\infty$ in \eqref{alfinito}, by using the Dominated Convergence Theorem on the right-hand side and Fatou's Lemma on the left-hand side. This finally gives \eqref{filippica} for a generic function $\phi\in\ddot{\mathrm{W}}^{1,p}_{q-1}(\Omega)$.

\section{Some consequences}
\label{sec:4}
In this section, we discuss some functional inequalities which are contained in nuce in Theorem \ref{thm:mother}. 
\subsection{General convex sets} We start with the following inequality, valid for general convex sets. We observe again that it is not necessary to assume $\phi\in L^q(\Omega)$.
\begin{corollary}
Let $1<p<\infty$ and $1<q< p$. Let $\Omega\subset\mathbb{R}^N$ be an open convex set. For every $\phi\in \ddot{\mathrm{W}}^{1,p}_{q-1}(\Omega)$ such that
\[
\int_\Omega |x|^\frac{p}{p-q}\,|\phi|^{q-1}\,dx<\infty, 
\] 
we have
\begin{equation}
\label{filippicabis}
\left(\int_\Omega |\phi|^q\,dx\right)^{p-q+1}\le 2\,\left(\inf_{x_0\in\Omega}\int_\Omega |x-x_0|^\frac{p}{p-q}\,|\phi|^{q-1}\,dx\right)^{p-q}\,\int_\Omega |\nabla \phi|^p\,dx.
\end{equation}
\end{corollary}
\begin{proof}
Let $\phi$ be a function as in the statement. We use the notations of Theorem \ref{thm:mother} and take $\gamma_{opt}\in\Pi(\rho_0,\rho_1)$ an optimal transport plan for $W_r(\rho_0,\rho_1)$ (where, as usual, $r=p/(p-q)$). By using the triangle inequality we get
\[
\begin{split}
W_{r}(\rho_0,\rho_1)&\le\left( \int_{\Omega\times\Omega} |x-x_0|^r\,d\gamma_{opt}\right)^{1/r}+\left( \int_{\Omega\times\Omega} |y-x_0|^r\,d\gamma_{opt}\right)^{1/r}\\
&=\left( \int |x-x_0|^r\,d\rho_0\right)^{1/r}+\left( \int |y-x_0|^r\,d\rho_1\right)^{1/r},
\end{split}
\]
for every $x_0\in\Omega$.
By using concavity of the map $\tau\mapsto\tau^{1/r}$, this in turn gives
\begin{equation}
\label{vaserstein}
\begin{split}
W_{r}(\rho_0,\rho_1)&\le 2^\frac{q}{p}\,\left(\int_\Omega |x-x_0|^r\,(d\rho_0+d\rho_1)\right)^{1/r}\\
&=2\,\left(\int_\Omega |x-x_0|^r\,|\phi|^{q-1}\,dx\right)^{1/r}\,\left(\int_\Omega |\phi|^{q-1}\,dx\right)^\frac{q-p}{p},
\end{split}
\end{equation}
where we used again \eqref{smezzo}, by assumption.
By using \eqref{vaserstein} in \eqref{filippica} and using the arbitrariness of $x_0\in\Omega$, we get the desired result.
\end{proof}
\subsection{Bounded convex sets}
In this case, Theorem \ref{thm:mother} implies some known inequalities, with explicit constants depending on simple geometric quantities and $p$ only.
\begin{corollary}[Nash-type inequality]
Let $1<p<\infty$ and $1<q< p$. Let $\Omega\subset\mathbb{R}^N$ be an open and bounded convex set. Then for every $\phi\in \ddot{\mathrm{W}}^{1,p}_{q-1}(\Omega)$ 
\begin{equation}
\label{filippica2}
\left(\displaystyle\int_\Omega |\phi|^q\,dx\right)^{p-q+1}\le \frac{\mathrm{diam}(\Omega)^p}{2^{p-1}}\,\int_\Omega |\nabla \phi|^p\,dx\,\left(\int_\Omega |\phi|^{q-1}\,dx\right)^{p-q}.
\end{equation}
\end{corollary}
\begin{proof}
In order to prove \eqref{filippica2}, it is sufficient to observe that for a bounded set we have 
\[
W_r(\rho_0,\rho_1)\le \mathrm{diam}(\Omega).
\] 
If we spend this information in \eqref{filippica}, we can then conclude.
\end{proof}
\begin{corollary}[Poincar\'e-Wirtinger inequality]
Let $1<p<\infty$ and $1<q< p$. Let $\Omega\subset\mathbb{R}^N$ be an open and bounded convex set. Then for every $\phi\in \mathrm{W}^{1,p}_{q-1}(\Omega)$, 
there holds
\begin{equation}
\label{filippica3}
\min_{t\in\mathbb{R}}\left(\displaystyle\int_\Omega |\phi-t|^q\,dx\right)^\frac{p}{q}\le \frac{\mathrm{diam}(\Omega)^p}{2^{p-1}}\,|\Omega|^{\frac{p}{q}-1}\,\int_\Omega |\nabla \phi|^p\,dx\,.
\end{equation}
\end{corollary}
\begin{proof}
Let $\phi\in \mathrm{W}^{1,p}_{q-1}(\Omega)$, by Lemma \ref{lm:dentro} we know in particular that $\phi\in L^q(\Omega)$. Then we can define $t_q$ the unique minimizer of 
\[
t\mapsto\left(\displaystyle\int_\Omega |\phi-t|^q\,dx\right)^\frac{p}{q}.
\]
By minimality, we have
\[
\int_\Omega |\phi-t_q|^{q-2}\,(\phi-t_q)\,dx=0.
\]
Thus the function $\phi-t_q$ belongs to $\ddot{\mathrm{W}}^{1,p}_{q-1}(\Omega)$. We just need to observe that since $\phi-t_q\in L^q(\Omega)$, then
\[
\left(\int_\Omega |\phi-t_q|^{q-1}\,dx\right)^{p-q}\le |\Omega|^\frac{p-q}{q}\,\left(\int_\Omega |\phi-t_q|^q\,dx\right)^{\frac{p-q}{q}\,(q-1)}.
\]
By using this in \eqref{filippica2} for the function $\phi-t_q$, we get the conclusion.
\end{proof}
\begin{remark}
Observe that the constant in \eqref{filippica3} degenerates to $0$ as the measure $|\Omega|$ gets smaller and smaller. This behaviour is optimal, as one may easily verify. Indeed, by taking $n\in\mathbb{N}\setminus\{0\}$ and
\begin{equation}
\label{parallelepipedi}
\Omega_n=[0,1]\times\left[0,\frac{1}{n}\right]\times\dots\times\left[0,\frac{1}{n}\right]\qquad \mbox{ and }\qquad \phi(x)=x_1,
\end{equation}
we have
\[
\frac{\displaystyle\left(\int_{\Omega_n} |\phi|^q\,dx\right)^\frac{p}{q}}{\displaystyle\int_{\Omega_n} |\nabla \phi|^p\,dx}\simeq \left(\frac{1}{n}\right)^{(N-1)\,\frac{p-q}{q}}=|\Omega_n|^\frac{p-q}{q}.
\]
\end{remark}
We conclude this list with an application to spectral problems. 
Let $1<p<\infty$, for every $\Omega\subset\mathbb{R}^N$ open and bounded set we introduce its {\it first non-trivial Neumann eigenvalue of the $p-$Laplacian}, i.e.
\[
\mu(\Omega;p):=\inf_{\phi\in \mathrm{W}^{1,p}(\Omega)\setminus\{0\}}\left\{\frac{\displaystyle\int_\Omega |\nabla \phi|^p\,dx}{\displaystyle\int_\Omega |\phi|^p\,dx}\,:\,\int_\Omega |\phi|^{p-2}\,\phi\,dx=0\right\}.
\]
The terminology is justified by the fact that for a connected set with Lipschitz boundary, the costant $\mu(\Omega;p)$ is the smallest number different from $0$ such that the Neumann boundary value problem
\[
\left\{\begin{array}{rclc}
-\mathrm{div}(|\nabla u|^{p-2}\,\nabla u)&=&\mu\,|u|^{p-2}\,u,& \mbox{ in }\Omega,\\
\displaystyle\frac{\partial u}{\partial \nu_\Omega}&=&0, & \mbox{ on }\partial\Omega
\end{array}
\right.
\]
admits non-trivial weak solutions. We then have the following result, which corresponds to the limit case $q=p$ of Theorem \ref{thm:mother}.
\begin{corollary}[Payne-Weinberger type estimate]
\label{coro:PW}
Let $1<p<\infty$ and let $\Omega\subset\mathbb{R}^N$ be an open and bounded convex set. We have the lower bound
\begin{equation}
\label{PW}
\left(\frac{2^\frac{p-1}{p}}{\mathrm{diam}(\Omega)}\right)^p\le \mu(\Omega;p).
\end{equation}
\end{corollary}
\begin{proof}
We take $\phi\in \mathrm{W}^{1,p}_{p}(\Omega)\setminus\{0\}$ such that $\int_\Omega |\phi|^{p-2}\,\phi\,dx=0$. Then we have
\begin{equation}
\label{00}
\min_{t\in\mathbb{R}} \displaystyle\int_\Omega |\phi-t|^p\,dx=\int_\Omega |\phi|^p\,dx.
\end{equation}
For $1<q<p$, we take $t_q\in\mathbb{R}$ to be the unique minimizer of 
\[
t\mapsto\left(\displaystyle\int_\Omega |\phi-t|^q\,dx\right)^\frac{p}{q}.
\]
By minimality of $t_q$ and Minkowski inequality, we have
\[
t_q\,|\Omega|^\frac{1}{q}-\left(\displaystyle\int_\Omega |\phi|^q\,dx\right)^\frac{1}{q}\le \left(\displaystyle\int_\Omega |\phi-t_q|^q\,dx\right)^\frac{1}{q}\le \left(\displaystyle\int_\Omega |\phi|^q\,dx\right)^\frac{1}{q}.
\]
This shows that $\{t_q\}_{q<p}$ is bounded, thus if we take the limit as $q$ goes to $p$, then $t_q$ converges (up to a subsequence) to some $\overline t$. By passing to the limit in \eqref{filippica3} we get
\[
\int_\Omega |\phi-\overline t|^p\,dx\le \frac{\mathrm{diam}(\Omega)^p}{2^{p-1}}\,\int_\Omega |\nabla \phi|^p\,dx.
\]
By keeping into account \eqref{00}, we get the desired conclusion.
\end{proof}
\begin{remark}
As mentioned in the Introduction, the constant appearing in the left-hand side of \eqref{PW} is not sharp. Indeed, the sharp lower bound is known to be
\begin{equation}
\label{PWs}
\left(\frac{\pi_p}{\mathrm{diam}(\Omega)}\right)^p< \mu(\Omega;p),\qquad \mbox{ where }\ \pi_p=2\,\pi\,\frac{(p-1)^\frac{1}{p}}{p\,\sin\left(\frac{\pi}{p}\right)},
\end{equation}
as proved by Payne and Weinberger in \cite{PW} for $p=2$ (see also \cite{Be}). The general case $p\not =2$ has been proved in \cite{ENT,FNT}. We recall that \eqref{PWs} is sharp in the following sense: for every convex set $\Omega$ the inequality in \eqref{PWs} is strict and it becomes asymptotically an equality along the sequence \eqref{parallelepipedi}.
\par
In the limit case $p=1$, a related result can be found in \cite{AD}.
\end{remark}

\end{document}